\documentclass[11pt]{amsart}
\usepackage{amsfonts}
\usepackage{}
\usepackage{mathrsfs}
\usepackage{bbm}
\usepackage{amssymb}
\usepackage{indentfirst}
\usepackage{latexsym,amsfonts,amssymb,amsmath}
\newtheorem{theorem}{Theorem}[section]
\newtheorem{lemma}[theorem]{Lemma}

\theoremstyle{definition}
\newtheorem{definition}[theorem]{Definition}
\newtheorem{proposition}[theorem]{Proposition}
\theoremstyle{remark}
\newtheorem{remark}[theorem]{Remark}
\newtheorem{example}[theorem]{Example}

\begin{document}

\title
{on paratopological gyrogroups}
\author{Ying-Ying Jin}\thanks{}
\address{(Y.Y. Jin) School of Mathematics and Computational Science, Wuyi University, Jiangmen 529020, P.R. China} \email{yingyjin@163.com}

\author{Li-Hong Xie*}\thanks{* The corresponding author.}
\address{(L.H. Xie) School of Mathematics and Computational Science, Wuyi University, Jiangmen 529020, P.R. China} \email{xielihong2011@aliyun.com}


\thanks{
This work is supported by Natural Science Foundation of China (Grant Nos. 11861018, 11871379, 11526158), the Natural Science Foundation of Guangdong
Province under Grant (No.2018A030313063), The Innovation Project of Department of Education of Guangdong Province  (No:2018KTSCX231), and key project of Natural Science Foundation of Guangdong Province Universities (2019KZDXM025).}



\begin{abstract}
 The concept of gyrogroups is a generalization of groups which
do not explicitly have associativity. Recently, Atiponrat extended the idea of topological (paratopological) groups to topological (paratopological) gyrogroups. In this paper, we prove that every regular (Hausdorff) locally gyroscopic invariant paratopological gyrogroup $G$ is completely regular (function Hausdorff). These results improve theorems of Banakh and Ravsky for paratopological groups. Also, we extend the Pontrjagin conditions of (para)topological groups to (para)topological gyrogroups.
\end{abstract}

\maketitle

\section{Introduction}
The concept of gyrogroup was discovered by A. A. Ungar when he
studied the Einstein velocity addition \cite{Ung}.
A gyrogroup is a generalization of a group in the sense that it is a groupoid with an identity and inverses,
but the associative law is redefined by more general definitions which are the left gyroassociative law and
the left loop property.
The gyrogroup does not form a group since it is neither associative nor commutative.
Nevertheless, Ungar showed that gyrogroup is rich in structure and
encodes a group-like structure, namely the gyrogroup structure \cite{Ung}.

Many important characteristics of gyrogroups have been intensively studied in \cite{Fer1}, \cite{Fer2}, \cite{Suk1}, \cite{Suk2} and \cite{Suk3}.
T. Suksumran and K. Wiboonton have been studying some basic algebraic properties of gyrogroups, for example, the isomorphism theorems, Cayley's Theorem, Lagrange's Theorem, the
gyrogroup actions, etc. in \cite{Suk1}, \cite{Suk2} and \cite{Suk3}.
Most of these properties are similar to those in classical group theory.

As far as we known, Atiponrat \cite{Atip} is the first scholar who extended the idea of topological groups to topological gyrogroups as gyrogroups with a topology such that its binary operation is jointly continuous and the operation of taking the inverse is continuous.
 Some basic properties of topological gyrogroups are studied by  and studied by Atiponrat \cite{Atip}. Specially, Atiponrat \cite{Atip} discovered that for a topological gyrogroup, $T_0$ and $T_3$ are equivalent. It is worth noting that Z. Cai, S. Lin and W. He in \cite{Cai} proved that every Hausdorff first countable topological gyrogroup is metrizable.

Recently, Atiponrat and Maungchang \cite{Atip1} studied some separation axioms of paratopological gyrogroups. They proved that:

\begin{theorem}\cite[Theorem 3]{Atip1}\label{Th1}
Let $G$ be a micro-associative regular paratopological gyrogroup such
that for any neighborhoods $A, B$ of the identity $e$ of $G$, $\overline{A}\subseteq \text{Int~}\overline{B\oplus A}$. Then $G$ is completely regular.
\end{theorem}

Banakh and Ravsky \cite{Ban} showed that every regular paratopological group is completely regular.
Hence they gave a generalization to Banakh and Ravsky's result.

In this paper, we introduce the concept of locally gyroscopic invariantness in paratopological gyrogroup and prove that every regular (Hausdorff) locally gyroscopic invariant paratopological gyrogroup $G$ is completely regular (function Hausdorff). Also, we find that the condition``for any neighborhoods $A, B$ of the identity $e$ of $G$, $\overline{A}\subseteq \text{Int}\overline{B\oplus A}$" in Theorem \ref{Th1} can be dropped. Finally, we also extend the Pontrjagin conditions of paratopological groups to paratopological gyrogroups.

All spaces are not assumed to satisfy any separation axiom unless otherwise stated.

$\bullet$ A space $X$ is {\it regular} if for any element $x\in X$ and an open set $U$ containing $x$, there is a neighborhood $V$ of $x$ such that $\overline{V}\subseteq U$.

$\bullet$ A space $X$ is {\it completely regular} if for any element $x\in X$ and an open set $U$ containing $x$, there is a continuous function $f:X\rightarrow [0,1] $ such that $f(x)=0$ and $f^{-1}([0,1))\subseteq U$.

$\bullet$ A space $X$ is {\it functionally Hausdorff} if for any distinct points $x, y \in X$ there exists a continuous function $f : X \rightarrow \mathbb{R}$ such that $f (x)\neq f (y)$.

\section{Definitions and preliminaries}

Let $G$ be a nonempty set, and let $\oplus  : G  \times G \rightarrow G $ be a binary operation on $G $. Then the pair $(G, \oplus)$ is
called a {\it groupoid.}  A function $f$ from a groupoid $(G_1, \oplus_1)$ to a groupoid $(G_2, \oplus_2)$ is said to be
a groupoid homomorphism if $f(x_1\oplus_1 x_2)=f(x_1)\oplus_2 f(x_2)$ for any elements $x_1, x_2 \in G_1$.  In addition, a bijective
groupoid homomorphism from a groupoid $(G, \oplus)$ to itself will be called a groupoid automorphism. We will write $Aut (G, \oplus)$ for the set of all automorphisms of a groupoid $(G, \oplus)$.
\begin{definition}\cite[Definition 2.7]{Ung}\label{Def:gyr}
 Let $(G, \oplus)$ be a nonempty groupoid. We say that $(G, \oplus)$ or just $G$
(when it is clear from the context) is a gyrogroup if the followings hold:
\begin{enumerate}
\item[($G1$)] There is an identity element $e \in G$ such that
$$e\oplus x=x=x\oplus e \text{~~~~~for all~~}x\in G.$$
\item[($G2$)] For each $x \in G $, there exists an {\it inverse element}  $\ominus x \in G$ such that
$$\ominus x\oplus x=e=x\oplus(\ominus x).$$
\item[($G3$)] For any $x, y \in G $, there exists an {\it gyroautomorphism} $\text{gyr}[ x, y ] \in Aut( G,  \oplus)$ such that
$$x\oplus (y\oplus z)=(x\oplus y)\oplus \text{gyr}[ x, y ](z)$$ for all $z \in G$;
\item[($G4$)] For any $x, y \in G$, $\text{gyr}[ x \oplus y, y ] = \text{gyr}[ x, y ]$.
\end{enumerate}
\end{definition}

In this paper, $\text{gyr}[a,b]V$ denotes $\{\text{gyr}[a,b]v: v\in V\}$.

The following Proposition \ref{Pro:gyr} below summarizes some algebraic properties of gyrogroups
\begin{proposition}(\cite{Ung1, Ung})\label{Pro:gyr}
Let $(G,\oplus)$ be a gyrogroup and $a,b,c\in G$. Then
\begin{enumerate}
\item[(1)] $\ominus(\ominus a)=a$ \hfill{Involution of inversion}
\item[(2)] $\ominus a\oplus(a\oplus b)=b$ \hfill{Left cancellation law}
\item[(3)] \text{gyr}$[a,b](c)=\ominus(a\oplus b)\oplus(a\oplus(b\oplus c))$ \hfill{Gyrator identity}
\item[(4)] $\ominus(a\oplus b)=\text{gyr}[a,b](\ominus b\ominus a)$\hfill{\text{cf.~}$(ab)^{-1}=b^{-1}a^{-1}$}
\item[(5)] $(\ominus a\oplus b)\oplus \text{gyr}[\ominus a,b](\ominus b\oplus c)=\ominus a\oplus c$ \hfill{\text{cf.~}$(a^{-1}b)(b^{-1}c)=a^{-1}c$}
\item[(6)] $\text{gyr}[a,b]=\text{gyr}[\ominus b,\ominus a]$ \hfill{Even property}
\item[(7)] $\text{gyr}[a,b]=\text{gyr}^{-1}[b,a], \text{the inverse of gyr}[b,a]$ \hfill{Inversive symmetry}
\end{enumerate}
\end{proposition}

\begin{definition}\cite[Definition 1]{Atip}
A triple $( G, \tau,  \oplus)$ is called a {\it topological gyrogroup} if and only if
\begin{enumerate}
\item[(1)] $( G, \tau)$ is a topological space;
\item[(2)] $( G, \oplus)$ is a gyrogroup;
\item[(3)] The binary operation  $\oplus: G  \times G  \rightarrow G$ is continuous where $G \times G$ is endowed with the product topology
and the operation of taking the inverse $\ominus(\cdot ) : G  \rightarrow G $, i.e. $x \rightarrow \ominus x$, is continuous.
\end{enumerate}
\end{definition}

\begin{definition}\cite[Definition 2.9]{Ung1}
Let $(G,\oplus)$ be a gyrogroup with gyrogroup operation (or,
addition) $\oplus$. The gyrogroup cooperation (or, coaddition) $\boxplus$ is a second
binary operation in $G$ given by the equation
$$(\divideontimes)~~~~a\boxplus b=a\oplus \text{gyr}[a,\ominus b]b$$ for all $a, b\in G$.
The groupoid $(G, \boxplus)$ is called a cogyrogroup, and is said to
be the cogyrogroup associated with the gyrogroup $(G, \oplus)$.

Replacing $b$ by $\ominus b$ in $(\divideontimes)$, along with Identity $(\divideontimes)$ we have the identity
$$a\boxminus b=a\ominus \text{gyr}[a,b]b$$ for all $a, b\in G$, where we use the obvious notation, $a\boxminus b = a\boxplus(\ominus b)$.
\end{definition}

\begin{theorem}\cite[Theorem 2.14]{Ung1}
Let $(G,\oplus)$ be a gyrogroup with cooperation $\boxplus$ given by Definition 2.4.
 Then,$$a\oplus b=a\boxplus \text{gyr}[a,b]b.$$
\end{theorem}

If a triple $( G, \tau,  \oplus)$ satisfies the first two conditions and its binary operation is continuous, we call such
triple a {\it paratopological gyrogroup} \cite{Atip1}. Sometimes we will just say that $G$ is a topological gyrogroup (paratopo-logical gyrogroup) if the binary operation and the topology are clear from the context.

\begin{proposition}\label{Pro1}\cite[proposition 1]{Atip1}
Let $G$ be a paratopological gyrogroup and $A$ be an open set. Then $B\oplus A$ is open for each $B\subseteq G$.
\end{proposition}

To study the separation axioms in paratopological gyrogroups, Atiponrat and Maungchang \cite[Definition 7]{Atip1} introduced the the following concept:

\begin{definition}\cite[Definition 7]{Atip1}
 A paratopological gyrogroup $G$ micro-associative if for any neighborhood $U \subseteq G$ of the identity $e $,
there are neighborhoods $W \subseteq V \subseteq U$ of $e$ such that $a \oplus( b \oplus V ) = ( a \oplus b ) \oplus V$ for any $a, b  \in W$.
\end{definition}

They showed that for each neighborhood $U$ every micro-associative paratopological gyrogroup $G$ of the identity $e$, there are neighborhoods $W \subseteq V \subseteq U$ of $e$ such that $\text{gyr~}[V]\subseteq  V$ \cite[Lemma 2]{Atip1}. This motives us introduce the following concept which plays an important role in our study.

\begin{definition}\label{De1}
Let $G$ be a paratopological gyrogroup and $\mathscr{B}$ a local base at the identity $e$ of $G$. We say that $\mathscr{B}$ is {\it locally gyroscopic invariant} if there is an open neighborhood $U$ of $e$ such that gyr$[a,b]V\subseteq V$ for each $V\in \mathscr{B}$ and $a,b\in U$. A paratopological gyrogroup $G$ is called {\it locally gyroscopic invariant} if $G$ has a locally gyroscopic invariant base at the identity.
\end{definition}

\begin{remark}
 The M\"{o}bius gyrogroup $\mathbb{D}=\{z\in \mathbb{C}:|z|<1\}$ with the standard topology is a topological gyrogroup(see \cite[Example 2]{Atip1}). In fact, $\mathbb{D}$ is micro-associative (see \cite[Example 8]{Atip1}). Also, one can easily show that $\mathbb{D}$ is locally gyroscopic invariant. However, we do not what are different for micro-associative paratopological gyrogroups and locally gyroscopic invariant paratopological gyrogroups.

\end{remark}
\section{Complete regularity in paratopological gyrogroups}

In this section, we shall study the separation axioms in paratopological gyrogroups. By the Urysohn's technology we prove the following important result.

\begin{proposition}\label{Pro}
Let $G$ be a locally gyroscopic invariant paratopological gyrogroup. Then for any neighborhood $O\subseteq G$ of $e$ there exists a continuous function $f : G\rightarrow[0, 1]$ such that $f (e) = 0$ and $f^{-1}([0, 1)) \subseteq \overline{O}$.
\end{proposition}

\begin{proof}
Since $G$ is a locally gyroscopic invariant paratopological gyrogroup, according to Definition \ref{De1} $G$ has a locally gyroscopic invariant base $\mathscr{B}$ at the identity $e$ of $G$ and an open neighborhood $U$ of $e$ such that gyr$[a,b]V\subseteq V$ for each $V\in \mathscr{B}$ and $a,b\in U$. Without loss the generality, we can assume that $U\subseteq O$ and every element of $\mathscr{B}$ is contained in $U$. Then we can find a family $\{U_i:i\in\omega\}\subseteq \mathscr{B}$ of open neighborhoods of $e$ such that $U_{i+1}\oplus U_{i+1}\subseteq U_i$ and $U_0\oplus U_0\subseteq U\subseteq O$ for each $i\in \omega$ and gry$[a,b]U_i\subseteq U_i$ holds for each $a,b\in U$ and each $i\in \omega$.

Put $V(1) = U_0$, fix $n\in \omega$, and assume that open neighborhoods $V(m/2^{n})$
of $e$ are defined for each $m = 1, 2,..., 2^n$. Put then $V(1/2^{n+1}) = U_{n+1}$, $V(2m/2^{n+1}) =
V(m/2^n)$ for $m = 1,..., 2^n$, and $$(*)~~~~V((2m + 1)/2^{n+1}) = V(m/2^n) \oplus U_{n+1} = V(m/2^n)\oplus V(1/2^{n+1}),$$
for each $m =  1, 2,..., 2^n-1$. By Proposition \ref{Pro1}, one can easily show that $V(r)$ is an open neighborhood of $e$ for every positive
dyadic rational number $r \leq1$. We claim the following conditions are satisfied:
\smallskip
\begin{enumerate}
\item $V(m/2^n) \oplus V(1/2^n) \subseteq V((m + 1)/2^n)\subseteq U\subseteq O$, for all integers $m$ and $n$ satisfying $ n > 0$ and $2^n> m > 0$.
\smallskip
\item $V(r_1)\subseteq V(r_2)$ for any positive
dyadic rational numbers $r_1$ and $r_2$ with $r_1< r_2\leq 1$
\smallskip
\item $\overline{V(r_1)}\subseteq \text{Int~}\overline{V(r_2)}$ for any positive
dyadic rational numbers $r_1$ and $r_2$ with $r_1< r_2\leq 1$.
\end{enumerate}
\smallskip

{\bf The proof of (1)}. Let us prove (1) by induction on $n$.

If $n= 1$, then the only possible value for $m$ is also 1, and we have: $$V(1/2)\oplus V(1/2) = U_1\oplus U_1\subseteq U_0 = V(1)\subseteq U\subseteq O.$$

Assume that $(1)$ holds for some $n$. Let us verify it for $n+1$.

If $m=2k$ for some $k\in \omega$, then we have that
\begin{align*}
&V(m/2^{n+1}) \oplus V(1/2^{n+1})
\\&=V(2k/2^{n+1}) \oplus V(1/2^{n+1})
\\&=V((2k+1)/2^{n+1}) \quad\quad\text{by the definition~} (*)
\\&=V((m+1)/2^{n+1}).
\end{align*}
This implies that $V(m/2^{n+1}) \oplus V(1/2^{n+1})\subseteq V((m+1)/2^{n+1}).$
Also,
\begin{align*}
&V((m+1)/2^{n+1})
\\&=V(2k/2^{n+1}) \oplus V(1/2^{n+1})  \quad\quad\text{by the definition~} (*)
\\&\subseteq V(k/2^{n}) \oplus U_n  \quad\quad\quad\text{by~}V(1/2^{n+1})=U_{n+1}\subseteq U_n
\\&=V(k/2^{n}) \oplus V(1/2^{n})  \quad\quad\quad\text{by~}V(1/2^{n})=U_{n}
\\&=V((k+1)/2^{n})\subseteq U\subseteq O. \quad\quad\quad\text{by the inductive assumption}
\end{align*}
Assume now that $0 < m = 2k + 1 < 2^{n+1}$, for some integer $k$. Then
\begin{align*}
&V(m/2^{n+1})\oplus V(1/2^{n+1})
\\& = V((2k + 1)/2^{n+1})\oplus U_{n+1} \quad\quad\quad\text{by~}V(1/2^{n+1})=U_{n+1}
\\&= (V(k/2^n)\oplus U_{n+1})\oplus U_{n+1} \quad\quad\text{by the definition~} (*)
\\&=\bigcup_{a\in V(k/2^n),b\in U_{n+1}}a\oplus (b\oplus \text{gyr}[b,a]U_{n+1}) \quad\text{by right gyroassociative law}
\\&\text{by locally gyroscopic invariant and the inductive assumption~}V(k/2^n)\subseteq U\subseteq O,
\\&\text{clearly,~} U_{n+1}\subseteq U\subseteq O, \text{~then}
\\&\subseteq\bigcup_{a\in V(k/2^n),b\in U_{n+1}}a\oplus (b\oplus U_{n+1})
\\&=V(k/2^n)\oplus (U_{n+1}\oplus U_{n+1})
\\& \subseteq V(k/2^n)\oplus U_{n}
\\&= V(k/2^n)\oplus V(1/2^n)
\\&\subseteq V((k+1)/2^n)   \quad\quad\quad\text{by the inductive assumption}
\\&=V(2(k+1)/2^{n+1})
\\&=V((m+1)/2^{n+1})
\\&=V((k+1)/2^n)\subseteq U\subseteq O \quad \quad\quad\text{by the inductive assumption}
\end{align*}
 This completes the proof of (1).

 {\bf The proof of (2)}. If $n=1$, then the only possible value for $m$ is 1, and we have:
 $$V(1/2)=U_1\subseteq U_0=V(2/2).$$
 Hence, to proof (2) it is enough to prove that
 $$V(m/2^{n+1})\subseteq V((m+1)/2^{n+1})\quad\quad\quad(**) $$
 holds for each $n\geq1$ and $0<m, m+1\leq 2^{n}$.

 If $m=2k$ for some $k$, then
 \begin{align*}
 &V(m/2^{n+1})=V(2k/2^{n+1})
 \\&=V(k/2^{n})\subseteq V(k/2^{n})\oplus U_{n+1}
 \\&=V(k/2^{n})\oplus V(1/2^{n+1}) \quad\quad\quad \text{by~~}V(1/2^{n+1})=U_{n+1}
 \\&=V((2k+1)/2^{n+1}) \quad\quad\text{by the definition~} (*)
 \\&=V((m+1)/2^{n+1}).
 \end{align*}
 If $m=2k+1$ for some $k$, then
  \begin{align*}
 &V(m/2^{n+1})=V((2k+1)/2^{n+1})
 \\&=V(k/2^{n})\oplus U_{n+1} \quad\quad\text{by the definition~} (*)
 \\&\subseteq V(k/2^{n})\oplus U_{n}
 \\&=V(k/2^{n})\oplus V(1/2^{n}) \quad\quad\quad \text{by~~}V(1/2^{n})=U_{n}
 \\&\subseteq V((k+1)/2^{n}) \quad\quad\quad \text{by the fact }(1)
 \\&=V(2(k+1)/2^{n+1})
 \\&=V((m+1)/2^{n+1}).
 \end{align*}

 This completes the proof of (2).

 {\bf The proof of (3)}. Since $r_1< r_2$, we can find positive integers $n,m$ such that $r_1\leq m/2^n < (m + 1)/2^n\leq r_2$.
From the fact that $V(r_1)\subseteq V(r_2)$ if $0<r_1<r_2\leq 1$,
it follows that:
\begin{align*}
&\overline{V(r_1)}\subseteq \overline{V(m/2^n)} \quad \quad\quad\quad\quad\quad\text{by the fact~} (2)
\\&\subseteq \overline{V(m/2^n)}\oplus V(1/2^n)
\\&\subseteq \overline{V(m/2^n)\oplus V(1/2^n)} \quad \text{by~} \oplus \text{~being continuous on~} G\times G
\\&\subseteq\overline{V((m+1)/2^n)} \quad \quad\quad\quad\text{by the fact~} (1)
\\&\subseteq \overline{V(r_2)} \quad \quad\quad\quad\quad\quad\text{by the fact~} (2).
\end{align*}
Since the set $ V(1/2^n)=U_n$ is open, by \cite[Proposition 1]{Atip1}, the set $\overline{V(m/2^n)}\oplus V(1/2^n)$ is open and contains $\overline{V(r_1)}$, and therefore, we obtain that $\overline{V(r_1)}\subseteq \text{Int~}\overline{V(r_2)}$.

Now we define a real-valued function $f$ on $G$ as follows:

$$
f(y)=\left\{\begin {array}{r@{\quad
\quad}l}1& \text{if~} y\in  G\setminus \overline{V(1)},\\
\text{inf~}\{r:y\in \overline{V(r)}\}& \text{if~} y\in \overline{V(1)}.
\end {array}\right. $$

Clearly, the function $f$ is well-defined and $f(e)=0,$ $f^{-1}([0,1))\subseteq \overline{V(1)}=\overline{U_0} \subseteq \overline{O}$. Thus it is enough to prove the following:

(4) $f$ is a continuous function on $G$.

{\bf Proof of (4).}  Take any $y\in G$ and an open neighborhood $W$ of $f(y)$ in $[0,1]$.

(i) If $f(y)=1$, then one can find a positive
dyadic rational number $r< 1$ such that $[r,1]\subseteq W$. According to the definition of $f$ it is obvious that $y\in G\setminus \overline{V(r)}$ and $f(G\setminus \overline{V(r)})\subseteq [r,1]\subseteq W$, which implies that $f$ is continuous at $y$.

(ii) If $0<f(y)<1$, then one can find positive
dyadic rational numbers $r_1< r_2< 1$ such that $f(y)\in (r_1, r_2)\subseteq [r_1, r_2]\subseteq W$. By (3) and the definition of $f$, one can easily show that $y\in (\text{Int~}\overline{V(r_2)})\setminus \overline{V(r_1)}$. Indeed, $y\notin \overline{V(r_1)}$ is obvious. By $f(y)< r_2$ one can find a positive
dyadic rational numbers $r_3$ such that $f(y)<r_3< r_2$. Hence, $y\in \overline{V(r_3)})\subseteq \text{Int~}\overline{V(r_2)}$ by (3) and the definition of $f$. Therefore, the set $(\text{Int~}\overline{V(r_2)})\setminus \overline{V(r_1)}$ is an open neighbourhood of $y$ such that $f((\text{Int~}\overline{V(r_2)})\setminus \overline{V(r_1)})\subseteq  [r_1, r_2]\subseteq W$, which implies that $f$ is continuous at $y$.

(iii) If $f(y)=0$, then one can find a positive
dyadic rational number $r< 1$ such that $[0,r]\subseteq W$. By (3) and the definition of $f$, one can easily show that $y\in \text{Int~}\overline{V(r)}$ and $f(\text{Int~}\overline{V(r)})\subseteq [0,r]\subseteq W$, which implies that $f$ is continuous at $y$. Hence we have proved that  $f$ is continuous at each point of $G$.
\end{proof}

Now we obtain our main results:
\begin{theorem}
Every regular locally gyroscopic invariant paratopological gyrogroup $G$ is completely regular.
\end{theorem}

\begin{proof}
Since $G$ is homogeneous, it is enough to show that there is a continuous real-valued function which can separate the identity $e$ in $G$ from the closed set $F$ such that $e\notin F$. Take any open neighborhood $O$ of $e$. Since $G$ is regular, one can find an open neighborhood $U$ of $e$ such that $\overline{U}\subseteq O$. According to Proposition \ref{Pro}, we can find a continuous function $f : G \rightarrow[0, 1]$ such that $f (e) = 0$ and $f^{-1}([0, 1)) \subseteq \overline{U}\subseteq O$. This implies that $f(y)=1$ for each $y\in G\setminus O$. Thus $G$ is completely regular.
\end{proof}

\begin{theorem}
Every Hausdorff locally gyroscopic invariant paratopological gyrogroup $G$ is functionally Hausdorff.
\end{theorem}

\begin{proof}
Take any distinct points $x, y \in G$. Since $G$ is homogeneous, we can assume that $x$ is the identity $e$ in $G$. Since $G$ is Hausdorff, one can find an open neighborhood $O$ of $x$ such that $y\notin \overline{O}$. According to proposition \ref{Pro}, we can find a continuous function $f : G\rightarrow[0, 1]$ such that $f (x) = 0$ and $f^{-1}([0, 1)) \subseteq \overline{O}$. This implies that $f(y)=1\neq 0=f(x)$. Thus $G$ is functionally Hausdorff.
\end{proof}

By Proposition \ref{Pro1}, we can prove the following result:

\begin{lemma}\label{L}
Let $G$ be a paratopological gyrogroup and  $A\subseteq G$ be an open neighborhood of the identity $e$. Then $A\subseteq \text{Int~}\overline{B\oplus A}$ for any set $B\subseteq G$ containing $e$.
\end{lemma}

\begin{proof}
Indeed, one can easily show that $A\subseteq B\oplus A\subseteq\overline{B\oplus A}$. Since $A$ is open in $G$, $B\oplus A$ is open by Proposition \ref{Pro1}. Hence, we have that $A\subseteq B\oplus A\subseteq \text{Int~}\overline{B\oplus A}$.
\end{proof}

Indeed,  Atiponrat and Maungchang \cite{Atip1} have shown that the condition ``for any neighborhoods $A, B$ of the identity $e$ of $G$" can be replaced by ``for any open neighborhoods $A, B$ of the identity $e$ of $G$" in Theorem \ref{Th1}. Thus, by Lemma \ref{L} we have the following result:

\begin{theorem}
Every micro-associative regular (Hausdorff) paratopological gyrogroup is completely regular (function Hausdorff).
\end{theorem}

\section{properties of (para)topological gyrogroups}
In this section, we will show that every topological gyrogroup has many self-homeomorphisms.
Let $(G,\oplus)$ be a gyrogroup, $x\in G$ and $A,B\subseteq G$.
We write $A\oplus B=\{a\oplus b:a\in A, b\in B\}$, $x\oplus A = \{x\oplus a : a \in A\}$
and $A\oplus x = \{a\oplus x : a \in A\}$.

\begin{proposition}\cite[Proposition 3 and Corollary 5]{Atip}\label{Def:gyr}
Let $G$ be a topological gyrogroup, $x\in G$ and $A,B\subseteq G$.
\begin{enumerate}
\item[(1)] the inverse mapping $inv:G\rightarrow G$, where $inv(x)=\ominus x$ for every $x\in G$, is a homeomorphism;
\item[(2)] the left translation $L_{x}: G \rightarrow G$, where $L_{x}(y) = x \oplus y$ for every $y\in G$, is a
homeomorphism;
\item[(3)] the right translation $R_{x}: G \rightarrow G$, where $R_{x}(y) = y\oplus x$ for every $y\in G$, is a
homeomorphism;
\item[(4)] if $A$ is open in $G$, then $x\oplus A$, $A\oplus x$, $A\oplus B$, $B\oplus A$ and $\ominus A$ are all open in $G$.
\end{enumerate}
\end{proposition}

Recall that a topological space $X$ is said to be homogeneous if for every $x\in X$ and
every $y \in X$, there exists a homeomorphism $f$ of the space $X$ onto itself such that
$f(x)=y$.

In the following theorem, we characterize the families of subsets of a gyrogroup $G$ which can appear as neighborhood bases of the neutral element in some paratopological gyrogroup
and consider the Pontrjagin conditions, which are generalization of Ravsky's results \cite{Rav}.

\begin{theorem}\label{the}
Let $G$ be a Hausdorff paratopological gyrogroup and $\mathcal{U}$ an open base at the neutral element $e$ of $G$.
Then the following conditions hold:
\begin{enumerate}
\item[(1)] for every $U\in\mathcal{U}$, there exists an element $V\in \mathcal{U}$ such that $V\oplus V\subseteq U$;
\item[(2)] for every $U\in\mathcal{U}$, and every $x\in U$, there exists $V\in \mathcal{U}$ such that $x\oplus V\subseteq U$;
\item[(3)] for every $U\in\mathcal{U}$ and $x\in G$, there exists $V\in \mathcal{U}$ such that $\ominus x\oplus(V\oplus x)\subseteq U$;
\item[(4)] for $U, V\in\mathcal{U}$, there exists $W\in\mathcal{U}$ such that $W\subseteq U\cap V$;
\item[(5)] for every $U\in\mathcal{U}$ and $a,b\in G$, there exists an element $V\in \mathcal{U}$ such
that gyr$[a,b]V\subseteq U$;
\item[(6)] for every $U\in\mathcal{U}$ and $b\in G$, there exists an element $V\in \mathcal{U}$ such that $\bigcup_{v\in V}\text{gyr}[v,b]V\subseteq U$;
\item[(7)] $\{e\}=\bigcap_{U\in\mathcal{U}} (U\boxminus U)$.
\end{enumerate}

Conversely, let $G$ be a gyrogroup and let $\mathcal{U}$ be a family of subsets of $G$ satisfying conditions (1)-(7).
Then the family $\mathcal{B}_{\mathcal{U}}=\{a\oplus U:a\in G, U\in \mathcal{U}\}$ is a base for a Hausdorff topology $\mathcal{T}_{\mathcal{U}}$ on $G$.
With this topology, $G$ is a paratopological gyrogroup.
\end{theorem}

\begin{proof}
Let $U\in\mathcal{U}$.

(1) For $G$ is a paratopological gyrogroup, then $op_2:G\times G\rightarrow G$ defined by $op_2(x,y)=x\oplus y$ is continuous.
Because $e\oplus e=e$, and $U\in \mathcal{U}$, there exist neighborhood $O$ and $W$ of $e$ such that $O\oplus W\subseteq U$.
We choose $V\in \mathcal{U}$ such that $V\subseteq O\cap W$.
Then $V\oplus V\subseteq W$.

(2) Let $x\in U$.
We define $R_x:G\rightarrow G$ by $R_x(y)=x\oplus y$.
Since $R_x(e)=x$ and $R_x$ is continuous at $e$, there exists $V\in\mathcal{U}$ such that $x\oplus V=R_x(V)\subseteq U$.

(3) For every $x\in G$, we define left translation map $L_{\ominus x}:G\rightarrow G$ by $L_{\ominus x}(y)=\ominus x\oplus y$.
By the continuous of $L_{\ominus x}$, for every $U\in\mathcal{U}$, there exists the neighborhood $V'$ of $x$ such that $L_{\ominus x}(V')\subseteq U$,
that is $\ominus x\oplus V'=L_{\ominus x}(V')\subseteq U$.
We also define the right translation map $R_{x}:G\rightarrow G$ by $R_{x}(y)=y\oplus x$.
Then $R_{x}(e)=x$.
Because $R_{x}$ is continuous at $e$, for the neighborhood $V'$ of $x$, there exists $V\in \mathcal{U}$ such that $R_{x}(V)\subseteq V'$,
that is $V\oplus x=R_{x}(V)\subseteq V'.$
So we get $\ominus x\oplus (V\oplus x)\subseteq \ominus x\oplus V'\subseteq U$.

(4) It is clear since $\mathcal{U}$ is an open base at $e.$

(5) For every $a, b\in G$,
we define $f_{a,b}:G\rightarrow G$ by $f_{a,b}(x)=\text{gyr}[a,b] x$.
Since $f_{a,b}(e)=e$ and $f_{a,b}$ is continuous at $e$,
for every $U\in\mathcal{U}$, there exists $V\in\mathcal{U}$ such that $f_{a,b}(V)\subseteq U$, that is, gyr$[a,b] V\subseteq U$.

(6) Take $W\in\mathcal{U}$ such that $W\oplus W\subseteq U$.
Then $b\oplus(W\oplus W)$ is an open set containing $b$.
Since $G$ is a paratopological gyrogroup, one can find $V\in\mathcal{U}$ such that $$(b\oplus V)\oplus V\subseteq b\oplus(W\oplus W)\quad\quad\quad(*)$$
Note that $$(b\oplus V)\oplus V=b\oplus (V\oplus \bigcup_{v\in V}\text{gyr}[v,b]V)\quad\quad\quad(**)$$
By (*) and (**) we have
$b\oplus(V\oplus \bigcup_{v\in V}\text{gyr}[v,b]V)\subseteq b\oplus(W\oplus W)$, which means $V\oplus \bigcup_{v\in V}\text{gyr}[v,b]V\subseteq W\oplus W$.
So we can get $\bigcup_{v\in V}\text{gyr}[v,b]V\subseteq W\oplus W\subseteq U$.

(7) We assume that $G$ is Hausdorff.
If $\bigcap_{U\in\mathcal{U}} (U\boxminus U)\neq\{e\}$, then there is $x\in \bigcap_{U\in\mathcal{U}} (U\boxminus U)$ such that $x\neq e$.
Since $G$ is Hausdorff, there are an open set $V_1$ containing $x$ and a open set $V\in \mathcal{U}$ such that $V_1\cap V=\emptyset$. Since $V_1$ is a neighborhood of $x$, $x\oplus e=x$ and $G$ is a paratopological gyrogroup,  one can find $U\in\mathcal{U}$ such that $(x\oplus U)\subseteq V_1$, hence we have that $(x\oplus U)\cap V=\emptyset$.
We choose $W\in \mathcal{U}$ such that $W\subseteq U\cap V$.
Then $(x\oplus W)\cap W=\emptyset$, that is, $x\notin W\boxminus W$, which is a contradiction.

To prove the converse, let $\mathcal{U}$ be a family of subsets of $G$ such that conditions (1)-(7) hold.
Let $\mathcal{T}=\{W\subseteq G: \text{for every~} x\in W~\text{there exists~} U\in\mathcal{U}~\text{such that~} x\oplus U\subseteq W\}.$

{\bf Claim 1.} $\mathcal{T}$ is a topology on $G$.
It is clear that $G\in\mathcal{T}$ and $\emptyset\in\mathcal{T}$.
It also easy to see that $\mathcal{T}$ is closed under unions.
To show that $\mathcal{T}$ is closed under finite intersections, let $V,W\in\mathcal{T}$.
Let $x\in V\cap W$. Since $x\in V\in\mathcal{T}$ and $x\in W\in\mathcal{T}$, there exist $O,Q\in\mathcal{U}$ such that $x\oplus O\subseteq V$ and $x\oplus Q\subseteq W$.
From (5) it follows that there exists $U\in\mathcal{T}$ such that $U\subseteq O\cap Q$.
Then, we have $x\oplus U\subseteq V\cap W$.
Hence, $V\cap W\in\mathcal{T}$, and $\mathcal{T}$ is a topology on $G$.

{\bf Claim 2.} If $O\in \mathcal{U}$ and $g\in G$, then $g\oplus O\in\mathcal{T}$.

Take any $x\in g\oplus O$, then $\ominus g\oplus x\in O$.
By property (2), there exists $V'\in\mathcal{U}$ such that $\ominus g\oplus x\oplus V'\subseteq O$.
For $V'$ and $\ominus g, x\in G$, there exists $V\in\mathcal{U}$ such that $\text{gyr}[\ominus g,x]V\subseteq V'$ by condition (5).
So we have $\ominus g\oplus(x\oplus V)=(\ominus g\oplus x)\oplus \text{gyr}[\ominus g,x]V\subseteq O$,
that is $x\oplus V\subseteq g\oplus O$.
Hence $g\oplus O\in\mathcal{T}$.

{\bf Claim 3.} The family $\mathcal{B}_{\mathcal{U}}=\{a\oplus U:a\in G, U\in \mathcal{U}\}$ is a base for the topology $\mathcal{T}$ on $G$.

Indeed, it follows from Claim 2.

{\bf Claim 4.} The multiplication in $G$ is continuous with respect to the topology $\mathcal{T}$.

Let $a, b$ be arbitrary elements of $G$, and $O$ be any element of $\mathcal{T}$ such that $a\oplus b\in O$.
Then there exists $W\in\mathcal{U}$ such that $(a\oplus b)\oplus W\subseteq O$.
There exists $U\in\mathcal{U}$ such that $a\oplus b\oplus \text{gyr}[a,b]U\subseteq(a\oplus b)\oplus W$ by condition (5).
For $U$ there exists $U_1\in\mathcal{U}$ such that $U_1\oplus U_1\subseteq U$.
For $b$ and $U_1$ there exists $U_2\in\mathcal{U}$ such that $\bigcup_{v\in U_2}\text{gyr}[v,b]U_2\subseteq U_1$ by condition (6).
By condition (4), we can get $U_3\subseteq U_1\cap U_2$.
For $U_3\in\mathcal{U}$, apply (3) to choose $U_4\in\mathcal{U}$ such that $U_4\oplus b\subseteq b\oplus U_3$.
Using condition (6) we can get $U_5\in\mathcal{U}$ such that $\bigcup_{v\in U_5}\text{gyr}[v,b]U_5\subseteq U_3$.
By the condition (4), we get $U_6\subseteq U_4\cap U_5$.
We have
\begin{align*}
\\& a\oplus U_6\oplus (b\oplus U_6)) \quad
\\&= a\oplus U_6\oplus \text{gyr}[a,e](b\oplus U_6)) \quad
\\&\subseteq a\oplus U_6\oplus \bigcup_{v\in U_6}\text{gyr}[a,v](b\oplus U_6)) \quad
\\&= a\oplus (U_6\oplus (b\oplus U_6)) \quad
\\&= a\oplus ((U_6\oplus b)\oplus \bigcup_{v\in U_6}\text{gyr}[v,b]U_6))) \quad
\\&\subseteq a\oplus ((U_4\oplus b)\oplus \bigcup_{v\in U_5}\text{gyr}[v,b]U_5))) \quad
\\&\subseteq a\oplus ((U_4\oplus b)\oplus U_3)  \quad \quad\quad\quad\quad\quad
\\&\subseteq a\oplus ((b\oplus U_3)\oplus U_3)  \quad \quad\quad\quad\quad\quad
\\&= a\oplus (b\oplus (U_3\oplus \bigcup_{v\in U_3}\text{gyr}[v,b]U_3))) \quad
\\&\subseteq a\oplus (b\oplus (U_1\oplus \bigcup_{v\in U_2}\text{gyr}[v,b]U_2))) \quad
\\&\subseteq a\oplus (b\oplus (U_1\oplus U_1)) \quad
\\&\subseteq a\oplus (b\oplus U)  \quad \quad\quad\quad\quad\quad
\\&=a\oplus b\oplus \text{gyr}[a,b]U\quad\quad\quad
\\&\subseteq(a\oplus b)\oplus W \quad\quad\quad
\end{align*}
Since $U_6\in\mathcal{U}$, $a\oplus U_6, b\oplus U_6$ are the neighborhood of $a, b$.
Thus, the multiplications in $G$ is continuous with respect to the topology $\mathcal{T}$.
This proves Claim 4.

{\bf Claim 5.} The gyrogroup $G$ with topology $\mathcal{T}$ is Hausdorff.

For every $x,y\in G$ and $x\neq y$, then $\ominus y\oplus x\neq e$.
There exist $U\in\mathcal{U}$ such that $\ominus y\oplus x\notin U\boxminus U$, which implies $\ominus y\oplus x\oplus U\cap U=\emptyset$.
For $x,y\in G$ and $U\in\mathcal{U}$, there exists $V\in\mathcal{U}$ such that $\text{gyr}[\ominus y, x]V\subseteq U$ by condition (5).
Then we claim that $x\oplus V\cap y\oplus U=\emptyset$, which implies that the gyrogroup $G$ with the topology $\mathcal{T}$ is Hausdorff.

In fact, if  $x\oplus V\cap y\oplus U\neq \emptyset$, then $\ominus y \oplus(x\oplus V)\cap U\neq \emptyset$. Hence, we have that $(\ominus y \oplus x)\oplus \text{gyr}[\ominus y, x]V\cap U \neq \emptyset$. Since $\text{gyr}[\ominus y, x]V\subseteq U$, we have that $(\ominus y \oplus x)\oplus U\cap U \neq \emptyset$. This is a contradiction.

We prove that $G$ is a Hausdorff paratopological gyrogroup with the topology $\mathcal{T}$.


\end{proof}

\begin{lemma}\label{LEM}
Let $G$ be a topological gyrogroup and $x\in G$. Then $L_x^\boxplus(\cdot):G\rightarrow G$ is homeomorphisms, where $L_x^\boxplus(\cdot)$ is defined as: $L_x^\boxplus(y)=x\boxplus y$ for each $y\in G$.
\end{lemma}

\begin{proof}
According to the definition, we have that
\begin{align*}
x\boxplus y&=x \oplus \text{gyr}[x,\ominus y]y
\\&=x \oplus(\ominus(x\ominus y)\oplus(x\oplus(\ominus y \oplus y)))
\\&=x \oplus(\ominus(x\ominus y)\oplus x)
\end{align*}
Hence, $L_x^\boxplus(y)=L_x(R_x(\ominus(L_x(\ominus(y)))))$. Since the operations $L_x, R_x$ and $\ominus$ are homeomorphisms, so is their the compositions.
\end{proof}

\begin{theorem}\label{the1}
Let $G$ be a Hausdorff topological gyrogroup and $\mathcal{U}$ an open base at the neutral element $e$ of $G$.
Then the conditions (1)-(7) in Theorem \ref{the} hold. In addition, the following conditions hold:
\begin{enumerate}
\item[(8)] for every $U\in\mathcal{U}$ and $x\in G$, there exists $V\in \mathcal{U}$ such that $V\boxplus x\subseteq x\oplus U$ and $x \oplus V \subseteq x\boxplus U$;
\item[(9)] for every $U\in\mathcal{U}$, there exists $V\in \mathcal{U}$ such that $\ominus V\subseteq U$.
\end{enumerate}

Conversely, let $G$ be a gyrogroup and let $\mathcal{U}$ be a family of subsets of $G$ satisfying conditions (1)-(9).
Then the family $\mathcal{B}_{\mathcal{U}}=\{a\oplus U:a\in G, U\in \mathcal{U}\}$ is a base for a Hausdorff topology $\mathcal{T}_{\mathcal{U}}$ on $G$.
With this topology, $G$ is a topological gyrogroup.
\end{theorem}

\begin{proof}
Let $G$ be a Hausdorff topological gyrogroup and $\mathcal{U}$ an open base at the neutral element $e$ of $G$. Then the conditions (1)-(7) hold by Theorem \ref{the}.

For (8). Since $G$ is a topological gyrogroup,  it is obvious that $(x\oplus U)\ominus x$ is an open set containing the the neutral element $e$ of $G$.  Hence there is a $V_1\in \mathcal{U}$ such that $V_1\subseteq (x\oplus U)\ominus x$, which is equivalent to $V_1\boxplus x\subseteq x\oplus U$. By Lemma \ref{LEM}, we have that $x\boxplus U$ is an open set containing $x$. Since the operation $L_x$ is continuous and $L_x(e)=x$, one can find $V_2\in \mathcal{U}$ such that $L_x(V_2)=x \oplus V_2 \subseteq x\boxplus U$. Take a $V\in \mathcal{U}$ such that $V\subseteq V_1\cap V_2$. Then the set $V$ is the required.

For (9). Since $G$ is a topological gyrogroup, the operation $\ominus$ is continuous. Clearly, $U$ is an open set containing the neutral element $e$, so one can find $V\in\mathcal{U}$ such that $\ominus V\subseteq U$.

To prove the converse, let $\mathcal{U}$ be a family of subsets of $G$ such that conditions (1)-(9) hold.
Let $\mathcal{T}=\{W\subseteq G: \text{for every~} x\in W~\text{there exists~} U\in\mathcal{U}~\text{such that~} x\oplus U\subseteq W\}.$ Then by Theorem \ref{the}
we have that $G$ with the topology $\mathcal{T}$ is a paratopological gyrogroup. Thus It is enough to show that the inverse operation $\ominus: (G, \mathcal{T})\rightarrow (G, \mathcal{T})$ is continuous.

Take any $x\in G$ and any $U\in \mathcal{U}$. By the condition (8), there is $U_1\in \mathcal{U}$ such that $U_1\boxplus(\ominus x)\subseteq \ominus x\oplus U$. For $U_1$, applying the condition (9), one can find $U_2\in \mathcal{U} $ such that $\ominus U_2\subseteq U_1$. For $U_2$, applying the condition (8) again, one can find $V\in  \mathcal{U}$ such that $x\oplus V\subseteq x \boxplus U_2$. Then we have that
\begin{align*}
\ominus(x\oplus V)&\subseteq \ominus(x \boxplus U_2)
\\&=\ominus U_2 \boxplus (\ominus x)
\\&\subseteq U_1\boxplus (\ominus x)
\\&\subseteq \ominus x\oplus U
\end{align*}
Thus we have proved that the inverse operation $\ominus$ is continuous.
\end{proof}

Clearly, every topological gyrogroup $G$ is a paratopological gyrogroup.
In general, paratopological gyrogroups need not be topological gyrogroups.
There exists a gyrogroup $G$ and a family $\mathcal{U}$ of subsets of $G$ satisfying conditions (1)-(7) of Theorem\ref{the}.
And the family $\mathcal{B}_{\mathcal{U}}=\{a\oplus U:a\in G, U\in \mathcal{U}\}$ is a base for a topology $\mathcal{T}_{\mathcal{U}}$ on $G$.
With this topology, $G$ is a paratopological gyrogroup.


\begin{example}
The M$\ddot{o}$bius gyrogroup equipped with the standard topology $\mathcal{T}$ is a
topological gyrogroup (see \cite[Example 2]{Atip}).

Put $\mathcal{U}=\{U_n:n\in \mathbb{N}\}$, where $U_n=\{x\in \mathbb{D}:|x|<\frac{1}{n}\}$ for each $n\in \mathbb{N}$.
Then the family $\mathcal{B}_{\mathcal{U}}=\{a\oplus_M U_n:a\in \mathbb{D}, U_n\in \mathcal{U}\}$ is a base for a Hausdorff topology $\mathcal{T}_{\mathcal{U}}$ on $\mathbb{D}$.
With this topology, $\mathbb{D}$ is a topological gyrogroup. In particular, we have that $\mathcal{T}_{\mathcal{U}}=\mathcal{T}$.


\end{example}

\begin{proof}
Clearly, $\mathcal{U}$ is a local base at the identity 0 when the M$\ddot{o}$bius gyrogroup $\mathbb{D}$ and any $n \in\omega$, let $U_n = \{x \in \mathbb{D} : |x| <\frac{1}{n}\}$.
Let's verify $\mathcal{U} = \{U_n : n \in \omega\}$ satisfying conditions (1)-(9).

(1) For each $n\in\omega$, we put $m=[n+\sqrt{n^2+1}]+1$ such that
$|x\oplus y|=|\frac{x+y}{1+\overline{x}y}|\leq\frac{|x|+|y|}{1-|\overline{x}y|}\leq\frac{2/m}{1-1/m^2}<1/n$ for every $x,y \in U_m $,
that is $x\oplus y\subseteq U_n$. Thus we verify $U_m\oplus U_m\subseteq U_n$.

(2) For each $n\in\omega$ and $x\in U_n$, we put $m=[N]+1$, where $1/{N}=\frac{{1}/{n}-|x|}{1+({1}/{n})|x|}$, such that
for every $y\in U_m$, we have $|x\oplus y|=|\frac{x+y}{1+\overline{x}y}|\leq\frac{|x|+|y|}{1-|\overline{x}y|}<\frac{|x|+1/m}{1-|\overline{x}(1/m)|}<\frac{|x|+1/N}{1-|\overline{x}(1/N)|}
=\frac{|x|+\frac{{1}/{n}-|x|}{1+({1}/{n})|x|}}{1-|\overline{x}\frac{{1}/{n}-|x|}{1+({1}/{n})|x|}|}=1/n$.
So we get $x\oplus y\in U_n$,
that is $x\oplus U_m\subseteq U_n$.

(3) For each $n\in\omega$ and $x\in \mathbb{D}$,  we put $m=[T]+1$, where $1/{T}=\frac{{1}/{n}(1-|x|^2)}{\sqrt{|x|^4+(2-4/n^2)|x|^2+1}}$, such that
for every $v\in U_m$ and $x\in \mathbb{D}$, we have

$|\ominus x\oplus(v\oplus x)|=|\ominus x\oplus \frac{x+v}{1+\overline{v}x}|=|\frac{-x+\frac{x+v}{1+\overline{v}x}}{1-\overline{x}\frac{x+v}{1+\overline{v}x}}|$
$=|\frac{-\overline{v}x^2+v}{1+\overline{v}x-v\overline{x}-x\overline{x}}|<\frac{1}{n}$,

which implies $\ominus x\oplus(U_{m}\oplus x)\subseteq U_n$.

In fact, for the above $v$ and $x$ we have
$$(|x|^4+1-(2/n^2)|x|^2)|v|^2+2(1-1/n^2)|v|^2|x|^2<1/n^2(1-|x|^2)^2.$$
The above formula implies
$$(|x|^4+1-(2/n^2)|x|^2)|v|^2-(1-1/n^2)2Re(\overline{v}^2x^2)<1/n^2(1-|x|^2)^2$$
if and only if $$(-\overline{v}x^2+v)\overline{(-\overline{v}x^2+v)}<1/n^2(1+\overline{v}x-v\overline{x}-\overline{x}x)(\overline{1+\overline{v}x-v\overline{x}-\overline{x}x})$$
if and only if
$$|-\overline{v}x^2+v|^2<1/n^2|1+\overline{v}x-v\overline{x}-\overline{x}x|^2$$
if and only if $$\frac{|-\overline{v}x^2+v|}{|1+\overline{v}x-v\overline{x}-x\overline{x}|}<\frac{1}{n}.$$

(4) For $U_m, U_n\in\mathcal{U}$,  we put $n_1=\text{max}\{m, n\}+1$ such that $W_{n_1}\subseteq U_m\cap U_n$.

(5) For every $U_n\in\mathcal{U}$ and $a,b\in \mathbb{D}$,  we put $m=n+1$ such that for each $v\in U_m$ we have $|\text{gyr}[a,b]v|=|v|<\frac{1}{m}<\frac{1}{n}$,
that is $\text{gyr}[a,b]U_m\subseteq U_n$.

(6) For every $U_n\in\mathcal{U}$ and $b\in \mathbb{D}$, we put $m=n+1$ such that for each $v\in U_m$ we have $|\text{gyr}[v,b]v|=|v|<\frac{1}{m}<\frac{1}{n}$,
that is $\bigcup_{v\in U_m}\text{gyr}[v,b]U_m\subseteq U_n$.

(7) Let $v\in\bigcap_{n\in\omega}( U_n\boxminus U_n)$ and $v\neq 0$. We put $t=[\frac{|v|}{-1+\sqrt{1+|v|^2}}]+1$
such that $|x\boxminus y|=|x\ominus \text{gyr}[x,y]y|=|\frac{x-\frac{1+x\overline{y}}{1+\overline{x}y}y}{1-\overline{x}\frac{1+x\overline{y}}{1+\overline{x}y}y}|\leq\frac{|x|+|y|}{1-|\overline{x}y|}\leq\frac{2/t}{1-1/t^2}<|v|$
for every $x,y\in U_t$, since $|\text{gyr}[x,y]|=1$.
This implies that $v\notin U_t\boxminus U_t$, which is a contradiction.
Thus we verify $\bigcap_{n\in\omega}(U_n\boxminus U_n)=\{0\}$.

Then the family $\mathcal{B}_{\mathcal{U}}=\{a\oplus U_n:a\in \mathbb{D}, U_n\in \mathcal{U}\}$ is a base for the topology $\mathcal{T_U}$ on $\mathbb{D}$.
In fact, $\mathcal{T_U}$ is equal to the standard topology.  We know that $\mathbb{D}$ equipped with the standard topology is a
topological gyrogroup, which further confirms our conclusion.

(8) For each $n\in\omega$ and $x\in \mathbb{D}$, we put $m=\text{max}\{[T]+1, 3n\}$, where $1/{T}=\sqrt{\frac{-B+\sqrt{B^2-4AC}}{2A}}$, $A=(1-|x|^2)^2|x|^2$, $B=(1-|x|^2)^2(1-\frac{1}{n^2}|x|^2+\frac{2|x|}{3n})$ and $C=\frac{1}{n^2}(1-|x|^2)^2(\frac{2|x|}{3n}-1)$ such that
for every $y\in U_m$ and $x\in \mathbb{D}$, we have

$|\ominus x\oplus(y\boxplus x)|=|\ominus x\oplus(y\oplus \text{gyr}[y,\ominus x]x)|=|\ominus x\oplus(y\oplus \frac{1-\overline{x}y}{1-\overline{y}x}x)|
=|\ominus x\oplus\frac{(1-|x|^2)y+(1-|y|^2)x}{1-|x|^2|y|^2}|=$
$\frac{|(|x|^2-1)|y|^2x+(1-|x|^2)y|}{|(1-|x|^2)(1-\overline{x}y)|}<\frac{1}{n}$,
which implies $\ominus x\oplus(U_{m}\boxplus x)\subseteq U_n$, furthermore, $U_{m}\boxplus x\subseteq x\oplus U$.

In fact, for the above $y$ and $x$ we have

$(1-|x|^2)^2|y|^4+(1-|x|^2)^2(1-\frac{1}{n^2}|x|^2)|y|^2+(1-|x|^2)^2\frac{2|x|}{3n}|y|^2+\frac{1}{n^2}(1-|x|^2)^2\frac{2|x|}{3n}$
$$<\frac{1}{n^2}(1-|x|^2)^2$$
The above formula implies

$(1-|x|^2)^2|y|^4+(1-|x|^2)^2(1-\frac{1}{n^2}|x|^2)|y|^2-(1-|x|^2)^2\cdot2\text{Re}(\overline{x}y)|y|^2+\frac{1}{n^2}(1-|x|^2)^2\cdot2\text{Re}(\overline{x}y)$
$$<\frac{1}{n^2}(1-|x|^2)^2$$
if and only if

$[(|x|^2-1)|y|^2x+(1-|x|^2)y]\overline{[(|x|^2-1)|y|^2x+(1-|x|^2)y]}$
$$<1/n^2[(1-|x|^2)(1-\overline{x}y)](\overline{[(1-|x|^2)(1-\overline{x}y)]}$$
if and only if
$$|(|x|^2-1)|y|^2x+(1-|x|^2)y|^2<1/n^2|(1-|x|^2)(1-\overline{x}y)|^2$$
if and only if $$\frac{|(|x|^2-1)|y|^2x+(1-|x|^2)y|}{|(1-|x|^2)(1-\overline{x}y)|}<\frac{1}{n}.$$

For every $U_n\in\mathcal{U}$ and $x\in \mathbb{D}$, we get $V=\bigcup_{u\in U_n}\text{gyr}[x,\ominus x](\text{gyr}[x,\ominus u]u)$.
Since for every $v\in V$, $|v|=|\text{gyr}[x,\ominus x](\text{gyr}[x,\ominus u]u)|=|u|<1/n$, we have $V\in\mathcal{U}$.
Then, for every $v\in V$ and some $u\in U_n$ we have
\begin{align*}
\ominus x\oplus (x \oplus v)&=(\ominus x\oplus x) \oplus \text{gyr}[\ominus x,x]v
\\&=\text{gyr}[\ominus x,x]v
\\&=\text{gyr}[\ominus x,x](\text{gyr}[x,\ominus x](\text{gyr}[x,\ominus u]u))
\\&=\text{gyr}[x,\ominus u]u
\end{align*}
So we get $\ominus x\oplus (x \oplus v)=\text{gyr}[x,\ominus u]u$, that is $x \oplus v= x\oplus \text{gyr}[x,\ominus u]u=x\boxplus u$.
We have proved $x \oplus V\subseteq x\boxplus U$.

(9) For every $U\in\mathcal{U}$, we put $V=U$ such that $U=\ominus U=\ominus V\subseteq U$, since $\ominus U=U$.

At last, we prove $\mathcal{T}_{\mathcal{U}}=\mathcal{T}$.
In fact, for each $r\in (0,1)$ we can get $m=[T]+1$, where $1/{T}=\frac{-a(1-r^2)+r(1-a^2)}{1-r^2a^2}$,
such that
for every $u\in U_m$ and $a\in \mathbb{D}$, we have $|a\oplus u|=|\frac{a+u}{1+\overline{a}u}|<r$.

Since $$(1-r^2|a|^2)|u|^2+2(1-r^2)|a\overline{u}|+|a|^2-|r^2|<0$$ for every $u\in U_m$ and $a\in \mathbb{D}$.

The above formula implies
$$(1-r^2|a|^2)|u|^2+2(1-r^2)2\text{Re} a\overline{u}+|a|^2-|r^2|<0$$
if and only if $$(a+u)\overline{(a+u)}<r^2(1+\overline{a}u)\overline{(1+\overline{a}u)}$$
if and only if
$$|a+u|^2<r^2|1+\overline{a}u|^2$$
if and only if $$\frac{|a+u|}{|1+\overline{a}u|}<r.$$

We prove $a\oplus U_m\subseteq U(0,r)$, that is $\mathcal{T}_{\mathcal{U}}=\mathcal{T}$.
\end{proof}

\begin{lemma}
Let $G$ be a paratopological gyrogroup. For all $a,b\in G$ and every nontrivial normal subgyrogroup $H$ of $G$, the following holds:
$$\text{gyr}[a,b]H=H.$$
\end{lemma}
\begin{proof}
For $H$ is a nontrivial normal subgyrogroup,
there exists a gyrogroup homomorphisms $\varphi:G\rightarrow K$ between gyrogroups $G$ and $K$ such that $H=\text{ker}\varphi$.
For every $x\in\text{gyr}[a,b]\text{ker}\varphi$, there exists $y\in\text{ker}\varphi$ such that $x=\text{gyr}[a,b]y$.
So $\varphi(x)=\varphi(\text{gyr}[a,b]y)=\text{gyr}[\varphi(a),\varphi(b)]\varphi(y)=\text{gyr}[\varphi(a),\varphi(b)]0=0$ (see \cite[Proposition 5.1]{Suk3}),
which implies $x\in\text{ker}\varphi$.
Thus, we have proved $\text{gyr}[a,b]\text{ker}\varphi\subseteq\text{ker}\varphi$ for all $a,b\in G$.
By Proposition 2.6 \cite{Suk3}, we have $\text{gyr}[a,b]\text{ker}\varphi=\text{ker}\varphi$.
That is $\text{gyr}[a,b]H=H$ For all $a,b\in G$.
\end{proof}

\begin{example}
Let $G$ be a gyrogroup and $$\mathcal{U}=\{U: U\subseteq G, U \text{~is a nontrivial normal subgyrogroup of G}\}$$
Then $\mathcal{U}$ satisfying conditions (1)-(9) and the family $\mathcal{B}_{\mathcal{U}}=\{a\oplus U:a\in G, U\in \mathcal{U}\}$ is a base for the topology $\mathcal{T_U}$ on $G$.
With this topology, $G$ is a topological gyrogroup.
\end{example}
\begin{proof}
(1) For each $U\in\mathcal{U}$, we can take $V=U$ such that $V\oplus V=U\oplus U=U$ by the closure property of groups.

(2) For every $U\in\mathcal{U}$, and every $x\in U$, we can take $V=U$ such that $x\oplus V=x\oplus U=U$ by the closure property of groups.

(3) For every $U\in\mathcal{U}$ and $x\in G$, we can take $V=U$ such that $\ominus x\oplus(V\oplus x)=\ominus x\oplus(U\oplus x)=\ominus x\oplus(x\oplus U)=(\ominus x\oplus x)\oplus \text{gyr}[\ominus x, x]U=U$.

(4) For every $U, V\in\mathcal{U}$, we can take $W=U\cap V\in\mathcal{U}$ such that $W\subseteq U\cap V$, since intersection of normal subgroups is a normal subgroup.

(5) For every $U\in\mathcal{U}$ and $a,b\in G$, there exists $V=U$ such
that $\text{gyr}[a,b]V=\text{gyr}[a,b]U=U$, since the normal subgyrogroup is invariant under the gyroautomorphisms of $G$.

(6) For every $U\in\mathcal{U}$ and $b\in G$, there exists $V=U$ such that $\bigcup_{v\in V}\text{gyr}[v,b]V=\bigcup_{v\in U}\text{gyr}[v,b]U= U$.

(7) Let $v\in\bigcap_{U\in\mathcal{U}}( U\boxminus U)$ and $v\neq 0$.
Then we have $v\in\bigcap_{}( \text{ker}\varphi\boxminus \text{ker}\varphi)$ for every gyrogroup homomorphisms $\varphi:G\rightarrow H$ between gyrogroups $G$ and $H$, since $U$ is a nontrivial normal subgyrogroup.
We put $\varphi=\text{gyr}[a,\ominus v]$, where $a\in G$ and $a\neq v$. It is obvious that $v\notin \text{ker} \varphi$, which is a contradiction.
Thus we verify $\bigcap_{U\in\mathcal{U}}( U\boxminus U)=\{0\}$.

(8) For every $U\in\mathcal{U}$ and $x\in G$,
since $U$ is a nontrivial normal subgyrogroup,
 there exists a gyrogroup homomorphisms $\psi:G\rightarrow H$ between gyrogroups $G$ and $H$ such that $U=\text{ker}\psi$ and $\psi^{-1}(\psi(x))=x\oplus U$.
For every $u\in U$ and $x\in G$, we have $\psi(u\boxplus x)=\psi(u)\boxplus \psi(x)=\psi(x)$ (see \cite[Proposition 5.1]{Suk3}), which means that $u\boxplus x\in x\oplus U$.
So we have proved $U\boxplus x\subseteq x\oplus U$.

Next we prove for every $U\in\mathcal{U}$ and $x\in G$, $x\oplus U= x\boxplus U$.
Let $x\in G$ and $u\in U$. By Definition 2.4, $x\boxplus u=x\oplus\text{gyr}[x,\ominus u]u$.
Since $U$ is a normal subgyrogroup, $\text{gyr}[x,\ominus u]u\in U$ by Lemma 4.6. That is $x\boxplus u\in x\oplus U$.
This proves $x\boxplus U\subseteq x\oplus U$.

On the other hand,
let $x\in G$ and $u\in U$. By Theorem 2.5 $x\oplus u=x\boxplus\text{gyr}[x, u]u$.
By Lemma 4.6, $\text{gyr}[x,u]u\in U$. Hence, $x\oplus u\in x\boxplus U$.
So $x\oplus U\subseteq x\boxplus U$.
We have proved $x\oplus U= x\boxplus U$.

(9) For every $U\in\mathcal{U}$, we put $V=U$ such that $U=\ominus U=\ominus V\subseteq U$, since $U$ is a nontrivial normal subgyrogroup of $G$.
\end{proof}

Let $G$ be a gyrogroup and $\mathcal{U}=\{0\}$.
It is obvious that $\mathcal{U}$ satisfies conditions (1)-(9).
Then the family $\mathcal{B}_{\mathcal{U}}=\{a\oplus U:a\in G, U\in \mathcal{U}\}$ is a base for the topology $\mathcal{T_U}$ on $G$.
In fact, $\mathcal{T_U}$ is the discrete topology.
With this topology, $G$ is a topological gyrogroup.

In addition, let $G$ be a gyrogroup and $\mathcal{U}=\{G\}$.
$\mathcal{U}$ also satisfies conditions (1)-(9).
Then the family $\mathcal{B}_{\mathcal{U}}=\{a\oplus U:a\in G, U\in \mathcal{U}\}$ is a base for the topology $\mathcal{T_U}$ on $G$.
In fact, $\mathcal{T_U}$ is the trivial topology.
With this topology, $G$ is a topological gyrogroup.


\end{document}